\documentclass[12pt]{article}
\usepackage{amsthm}
\usepackage{amsmath}
\usepackage{enumerate}
\usepackage{amssymb}
\usepackage{latexsym}
\usepackage{amsfonts}
\usepackage{color}
\usepackage{mathrsfs}
\usepackage{epsfig}
\newtheorem{theorem}{\bf Theorem}[section]

\newtheorem{definition}[theorem]{\bf Definition}

\newtheorem{lemma}[theorem]{\bf Lemma}

\newsavebox{\savepar}

\begin{document}
	
	\title{Multiplicity of solutions to an elliptic problem with singularity and measure data}
	\author{Sekhar Ghosh$^{\dagger}$, Akasmika Panda$^\dagger$ \&  Debajyoti Choudhuri$^{\dagger, }$\footnote{Corresponding author: dc.iit12@gmail.com}\\
	\small{$^\dagger$Department of Mathematics, National Institute of Technology Rourkela, India}\\
		\small{Emails: sekharghosh1234@gmail.com \& akasmika44@gmail.com}}
	\date{}
	\maketitle
	
	\begin{abstract}
		\noindent In this paper, we prove the existence of multiple nontrivial solutions of the following equation.
		\begin{align*}
		\begin{split}
		-\Delta_{p}u & = \frac{\lambda}{u^{\gamma}}+g(u)+\mu~\mbox{in}\,\,\Omega,\\
		u & = 0\,\, \mbox{on}\,\, \partial\Omega,\\
		u&>0 \,\,\mbox{in}\,\,\Omega,
		\end{split}
		\end{align*}
		where $\Omega \subset \mathbb{R}^N$ is a smooth bounded domain with $N \geq 3$, $1 <  p-1 < q$ , $ \lambda>0$, $\gamma>0$, $g$ satisfies certain conditions, $\mu\geq 0$ is a bounded Radon measure.
	\end{abstract}
	\begin{flushleft}		
		\end{flushleft}

	\section{Introduction}
	Elliptic equations with singularity has gained a huge attention owing to its richness both from the theoretical and application point of view. Early traces of research pertaining to problems involving singularity can be found in \cite{Lazer}, where the authors have addressed the following problem.
	\begin{eqnarray}\label{refeq1}
	\begin{split}
	-\Delta u&= \frac{f(x)}{u^{\gamma}}~\text{in}~\Omega,\\
	u&= 0~\text{on}~\partial\Omega,
	\end{split}
	\end{eqnarray}
	where $\Omega$ is a strictly convex, bounded domain in $\mathbb{R}^N$ with $C^2$ boundary. The existence of a unique solution was guaranteed  iff $0<\gamma <3$. The authors in \cite{Lazer}, has also shown the existence of a solution in $C^1(\bar{\Omega})$, for $0<\gamma <1$. Haitao \cite{Hai} studied the perturbed singular problem 
	\begin{eqnarray}\label{refeq2}
	\begin{split}
	-\Delta u&= \cfrac{\lambda}{u^\gamma}+u^p~\&~u>0~\text{in}~\Omega,\\
	u&= 0~\text{on}~\partial\Omega,
	\end{split}
	\end{eqnarray}
	and guaranteed the existence of two weak solutions for $\lambda<\Lambda$, no solution for $\lambda>\Lambda$ and atleast one solution for $0<\gamma<1<p\leq\frac{N+2}{N-2}$ and some $\Lambda>0$. A further generalization to this problem can be found in \cite{gia}, where the existence of two solutions were shown for some $0<\gamma<1<p-1<q\leq p^*-1$. An important problem involving singularity in the literature can be found in the work due to Crandall et al \cite{cran}, where the authors have addressed the problem.
	\begin{eqnarray}\label{cran_prob}
	-\Delta u &=& f(u)~\text{in}~\Omega,\nonumber\\
	u&=&0~\text{on}~\partial\Omega,
	\end{eqnarray}
	where $f$ is a function with singularity near $0$. The authors in \cite{cran}, have shown the existence of a unique classical solution in $C^2(\Omega)\cap C(\bar{\Omega})$. Another noteworthy work is due to Giacomoni and Sreenadh \cite{giasri}, where the authors have investigated the following quasilinear and singular problem.
	\begin{align}\label{giasri_prob}
	\begin{split}
	-\Delta_{p} u&=\frac{\lambda}{u^\delta}+u^q~\text{in}~\Omega,\\
	u&=0~\text{on}~\partial\Omega,\\
	u&>0~\text{in}~\Omega,
	\end{split}
	\end{align}
	where $\Omega$ is a bounded domain in $\mathbb{R}^N$ with smooth boundary, $1<p-1<q$ and $\lambda,\delta>0$. The authors have shown the existence of weak solutions for small $\lambda>0$ in $W_0^{1,p}(\Omega)\cap C(\bar{\Omega})$ if and only if $\delta<2+\frac{1}{p-1}$. Further they have investigated the radial symmetry case, i.e. for $\Omega=B_R(0)$, where they have proved the global multiplicity of solutions in $C(\bar{\Omega})$ with $\delta>0$, $1<p-1<q$, by using shooting method. Readers interested in `singularity involving problem' can refer to \cite{Oliva, canino2, canino3, tali} and of late Panda et al. \cite{pgc}, who have investigated a problem involving singularity and a measure. Motivated by the work due to \cite{Bal}, which stemmed out from the work due to \cite{Arcoya}, by generalizing their result for the $p$-Laplacian, we will study the following problem.
	\begin{align}\label{e0}
	(P)\hspace{1 cm}\left\{ \begin{aligned}	-\Delta_p u&= \cfrac{\lambda}{u^\gamma}+g(u)+\mu~\text{in}~\Omega,\\
	u&= 0~\text{on}~\partial\Omega,\\
	u&>0~\text{in}~\Omega,
	\end{aligned}\right.
	\end{align}
	where $\Omega$ is a strictly convex, bounded domain in $\mathbb{R}^N$ with $C^2$ boundary, $N>2$, $1<p<N$, $\Delta_pu=\text{div}\{|\nabla u|^{p-2}\nabla u\}$, $\lambda>0$, $\gamma>0$ and $\mu$ is a bounded Radon measure. The function $g$ obeys certain growth conditions, i.e there exists some constants $C>0$ such that, $$C^{-1}t^{1+q}\leq t g(t)\leq Ct^{1+q},$$ where $p-1<q<\cfrac{N(p-1)}{N-p}$.
	\section{Notations and Definitions}
	 We will use the notations due to \cite{Evans}, to denote $W_0^{k,p}(\Omega)$, to be the space obtained by considering the closure of $C_c^{\infty}(\Omega)$ in the Sobolev space $W^{k,p}(\Omega)$ and $W_{loc}^{k,p}(\Omega)$ to be the local Sobolev space, which consists of functions $u$ such that for any compact $K\subset \Omega$, $u\in W^{k,p}(K)$. The H\"{o}lder Space is denoted by  $C^{k,\beta}(\bar{\Omega})$ with  $0<\beta\leq1$ (again a notation borrowed from $\cite{Evans}$), which consists of all functions $u\in C^k(\bar{\Omega})$ such that the norm $$\sum\limits_{|\alpha|\leq k}\sup|D^{\alpha}u|+\sup\limits_{x\neq y} \left\{\frac{|D^ku(x)-D^ku(y)|}{|x-y|^{\beta}}\right\}<\infty.$$ We will use the truncation functions for fixed $k > 0$,
	$$T_k(t)=\max\{-k,\min\{k,t\}\} ~\text{and}~G_k(t)=(|t|-k)^+sign(t)$$with $t\in \mathbb{R}$. Observe that $T_k(t)+G_k(t)=t$ for any $t\in \mathbb{R}$ and $k>0$.\\
	We denote $\mathbb{M}(\Omega)$ as the space of all finite Radon measures on $\Omega$. For every $\mu\in \mathbb{M}(\Omega)$, we define $$||\mu||_{\mathbb{M}(\Omega)}=\int_{\Omega}d|\mu|.$$ 
	We will use the Marcinkiewicz space $\mathcal{M}^q(\Omega)$ (or weak $L^q(\Omega)$)  defined for every $0 < q <\infty$, as the space of all measurable functions $f:\Omega\rightarrow\mathbb{R}$ such that the corresponding distribution functions satisfy an estimate of the form
	$$m(\{x\in \Omega:|f(x)|>t\})\leq \frac{C}{t^q}\hspace{0.4cm}t>0,\,C<\infty.$$
	Indeed, for bounded domain $\Omega$ we have $\mathcal{M}^q\subset \mathcal{M}^{\bar{q}}$ if  $q\geq \bar{q}$, for some fixed positive $\bar{q}$. Further, the following continuous embeddings holds
	\begin{equation}\label{marcin}
	L^q(\Omega)\hookrightarrow \mathcal{M}^q(\Omega)\hookrightarrow L^{q-\epsilon}(\Omega),
	\end{equation}
	for every $1<q<\infty$ and $0<\epsilon<q-1$. We will use this embedding result to show the existence of solutions. We now give the definition of convergence in the measure space.

\begin{definition}\label{measure}
	Let $(\mu_n)$ be the sequence of measurable functions in $\mathbb{M}(\Omega)$. We say $(\mu_n)$ converges to $\mu\in \mathbb{M}(\Omega)$ in the sense of measure $\cite{Folland}$ i.e. $\mu_n\rightharpoonup \mu$ in $\mathbb{M}(\Omega)$, if
	$$\int_\Omega f d\mu_n\rightarrow \int_\Omega f d\mu, \hspace{0.2cm}\forall f\in C_0(\Omega).$$
\end{definition}
\noindent In order to show the existence of solutions to the problem (\ref{e0}), we will consider the following sequence of problems $(P_n)$.

\begin{align}\label{e7}
(P_n)\hspace{1 cm}\left\{ \begin{aligned}	-\Delta_{p} u&=\frac{\lambda}{(u+\frac{1}{n})^\gamma}+ g(u)+\mu_n~\text{in}~\Omega,\\
u&=0~\text{on}~\partial\Omega,\\
u&>0~\text{in}~\Omega,
\end{aligned}\right.
\end{align}
whose solution will be denoted by $u_n$. The weak formulation to \eqref{e7} is defined as
\begin{equation}\label{weak}
\int_{\Omega} |\nabla u_n|^{p-2}\nabla u_n\cdot\nabla\phi dx=\lambda\int_{\Omega} \frac{\phi}{(u_n+\frac{1}{n})^\gamma}+\int_{\Omega} g(u_n)\phi dx+ \int_{\Omega} \mu_n\phi
dx, \,\forall\,\phi\in C_0^{1}(\bar{\Omega})
\end{equation}
where, ($\mu_n$) is a sequence of smooth non-negative functions bounded in $L^1(\Omega)$ and converging weakly to $\mu$ in the sense of Definition $\ref{measure}$. We now give the definition of weak solution to the problem ($P$) in \eqref{e0}.
\begin{definition}
	We say a function $u\in W_{loc}^{1,p}(\Omega)\cap L^{\infty}(\Omega)$ is a weak solution to the problem (\ref{e0}) if $\cfrac{\phi}{u^\gamma}\in L^1(\Omega)$ and it satisfies 	
	\begin{equation}\label{e1}
	\int_\Omega |\nabla u|^{p-2}\cdot\nabla u\cdot\nabla \phi~ dx=\lambda\int_\Omega \frac{\phi}{u^\gamma}dx +\int_\Omega g(u)\phi dx + \int_\Omega \phi d\mu
	\end{equation}
	for every $\phi\in W_0^{1,p}(\Omega^{'})$ with $\Omega^{'}\subset\subset \Omega$.
\end{definition}

\noindent In the subsequent section, we will prove a few lemmas which will be required to prove our main result in Section $\ref{main}$. Note that the solution will be named as $u_n$ in multiple places for different problems.
\section{Important Lemmas}\label{lemmas}
In this section we will prove a few important lemmas,  Lemma \eqref{l1} - \eqref{l6}, which are the main tools needed to prove the main result of existence of solution to the problem \eqref{e0}.
	\begin{lemma}\label{l1}
		The problem 
		\begin{eqnarray}\label{e2}
		\begin{split}
		-\Delta_p u&= \cfrac{\lambda}{(u+\frac{1}{n})^\gamma}~\text{in}~\Omega,\\
		u&=0~\text{on}~\partial\Omega,
		\end{split}
		\end{eqnarray}
		possesses a nonnegative weak solution in  $W_{loc}^{1,p}(\Omega)\cap L^{\infty}(\Omega)$ for each $n\in \mathbb{N}$.
	\end{lemma}
	\begin{proof}
		The idea of the proof is to apply Schauder's fixed point argument. For a fixed $n\in \mathbb{N}$ and a fixed $v\in L^p(\Omega)$, we define the map $J_\lambda:W_0^{1,p}(\Omega)\rightarrow \mathbb{R},$ as follows,
		$$J_\lambda(u)=\dfrac{1}{p}\int_\Omega |\nabla u|^p dx -\lambda\int_\Omega\cfrac{u}{(|v|+\frac{1}{n})^\gamma} dx.$$
		It is easy to see that, $J_\lambda$ is continuous, coercive and strictly convex in $W_0^{1,p}(\Omega).$ Therefore, the existence of a unique minimizer $w\in W_0^{1,p}(\Omega)$ corresponding to a $v\in L^p(\Omega)$ is certain.\\
		We define, $H: L^p(\Omega)\rightarrow L^p(\Omega)$ by $$H(v)= (-\Delta_p)^{-1} \left[\frac{\lambda}{(|v|+\frac{1}{n})^\gamma}\right]:=w.$$
		On choosing  $w$ as a test function from $W_0^{1,p}(\Omega)$ in the weak formulation of \eqref{e2}, we have
		\begin{align*}\label{e3}
		\begin{split}
		\int_\Omega|\nabla w|^p = \int_\Omega |\nabla w|^{p-2} \nabla w\cdot\nabla w&=\int_\Omega \cfrac{\lambda}{(|v|+\frac{1}{n})^\gamma} w \\
		&\leq\lambda n^{\gamma} \int_\Omega |w|. 
		\end{split}
		\end{align*}
Hence, by using the Poincar\'{e} inequality and the H\"{o}lder's inequality  on the left and right hand side respectively, we get 
		\begin{equation}\label{e4}
		\|w\|_p\leq C(n,\gamma,\lambda).
		\end{equation}
		Let us consider a sequence $(v_k)$ that converges to $v$ in $L^p(\Omega)$. By using the dominated convergence theorem, we have  \begin{center}
			$\left\|\cfrac{\lambda}{(|v_k| +\frac{1}{n})^\gamma}-\cfrac{\lambda}{(|v| +\frac{1}{n})^\gamma}\right\|_{L^p(\Omega)}\longrightarrow 0.$ 
		\end{center}	
		Thus, the convergence of $w_k=H(v_k)$  to $w=H(v)$ in $L^p(\Omega)$ can be followed from the uniqueness of the weak solution. Hence, the continuity of $H$  over $L^p(\Omega)$ is followed. By the estimate in equation $\eqref{e4}$ and by the Rellich-Kondrochov theorem, we get that $H(L^p(\Omega))$ is relatively compact in $L^p(\Omega)$. We now can apply the Schauder's fixed point theorem to guarantee the existence of a fixed point say $w$. By the regularity theorem of Lieberman \cite{Lieberman}, we have $u_n\in C^1(\bar{\Omega}), \forall n\in \mathbb{N}$. Using the strong maximum principle \cite{Guedda}, we have $w>0$ in $\Omega$ and this concludes the proof.
	\end{proof}
	\begin{lemma}\label{l2}
		The sequence $(u_n)$ is increasing w.r.t n and for every $K\subset\subset\Omega$, there exists $C_K$ (only depends on $K$) such that $u_n\geq C_K>0$, a.e. in $K$ with $||u_n||_\infty \leq R {\lambda}^{\frac{1}{\gamma+p-1}},~\forall n\in\mathbb{N}$, $R$ is independent of $n$.
	\end{lemma}
	\begin{proof}
		Consider a sequence of problems
		\begin{align}\label{e5}
		\begin{split}
		-\Delta_p u&= \cfrac{\lambda}{(u +\frac{1}{n})^\gamma} ~\text{in}~\Omega,\\
		u&=0~\text{on}~\partial\Omega.\\
		\end{split}
		\end{align}
		For each $n$, let $u_n$ be the solution to the problem \eqref{e5}. Consider,
		$$	\int_\Omega (|\nabla u_n|^{p-2}\cdot\nabla u_n-|\nabla u_{n+1}|^{p-2}\cdot\nabla u_{n+1})\cdot\nabla\phi~ dx=\lambda\int_\Omega \left((u_n +\frac{1}{n})^{-\gamma}-(u_{n+1} +\frac{1}{n+1})^{-\gamma}\right)\phi~ dx.$$
		We choose, the test function $\phi=(u_n-u_{n+1})^+$ to obtain,
		$$\int_\Omega (|\nabla u_n|^{p-2}\cdot\nabla u_n-|\nabla u_{n+1}|^{p-2}\cdot\nabla u_{n+1})\cdot \nabla(u_n-u_{n+1})^+ dx$$
		$$ \hspace{6cm}	\leq \lambda\int_\Omega \left((u_n +\frac{1}{n+1})^{-\gamma}-(u_{n+1} +\frac{1}{n+1})^{-\gamma}\right)(u_n-u_{n+1})^+ dx.$$
		Using the inequalities from \cite{grey}, we get for $p\geq2$, 
		\begin{align*}
		\int_\Omega (|\nabla u_n|^{p-2}\nabla u_n-|\nabla u_{n+1}|^{p-2}\cdot\nabla u_{n+1})\cdot \nabla(u_n-u_{n+1})^+ dx &\geq C_p ||\nabla(u_n-u_{+n+1})^+||^p\\
		&\geq 0
		\end{align*}
		and for $1<p<2$,
		\begin{align*}
		\int_\Omega (|\nabla u_n|^{p-2}\nabla u_n-|\nabla u_{n+1}|^{p-2}\nabla u_{n+1})\cdot \nabla(u_n-u_{n+1})^+ dx &\geq C_p \frac{||u_n-u_{+n+1}||^2}{(||u_n||+||u_{+n+1}||)^{2-p}}\\
		&\geq 0.
		\end{align*}
		\noindent Therefore, we have
		$$0 \leq \lambda\int_\Omega \left\{\left( u_n +\frac{1}{n+1}\right)^{-\gamma}-\left( u_{n+1} +\frac{1}{n+1}\right)^{-\gamma}\right\}(u_n-u_{n+1})^+ dx \leq 0.$$
		Hence, we get $||(u_n-u_{n+1})^+||=0$. This implies $u_n$ is monotonically increasing w.r.t $n$. Now, using the Strong Maximum principle \cite{Vazquez}, we get $u_1>0$ in $\Omega$, where $u_1$ is the solution of the problem $\eqref{e5}$ with $n=1$. Since, $u_n$ is monotonically increasing with respect to $n$, we have $u_n>u_1$ in $\Omega$ and hence we conclude that $u_n>C_K>0$, for every $K\subset\subset\Omega$ with $C_K$ being independent of $n$.\\
		$\textit{Claim:}$ $(u_n)$ is uniformly bounded in $\Omega.$\\
	{\it Case 1:} When $\lambda=1$. Define, $M(k)=\{x\in\Omega:u_n>k\}$ and $$S_k(u_n)= { \left\{
			\begin{array}{ll}
			u_n-k; & \text{if}~u_n>k	\\
			0; &\text{if}~u_n\leq k.
			\end{array}
			\right. }$$
		We choose, $S_k(u_n)$ as the test function in the weak formulation of \eqref{e5} to get,
		\begin{align}\label{e6}
		\int_{M(k)}|\nabla u_n|^{p-2}\nabla u_n.\nabla u_n & = \int_{M(k)}|\nabla u_n|^p
		\nonumber\\& =  \int_{M(k)}\frac{u_n-k}{(u_n+\frac{1}{n})^\gamma}\nonumber\\
		 &<\int_{M(k)}\frac{u_n-k}{u_n^\gamma}\nonumber
		\\&\leq ||u_n-k||_{L^p(M(k))} |M(k)|^{\frac{1}{p\prime}}\nonumber\\& \leq C ||\nabla u_n||_{L^p(M(k))}|M(k)|^{\frac{1}{p\prime}},~~~~\text{(  by the  Poincar\'{e} inequality)}.\nonumber \end{align}
		By using the Sobolev embedding theorem, we get
		$$||u_n||^{p-1}_{L^{p^*}(M(k))}< \frac{C}{S^{p-1}}|M(k)|^{\frac{1}{p\prime}},~ \text{where}~ p^*=\frac{Np}{N-p}~\text{(Sobolev conjugate of}~ p).$$
		It is easy to see that, for $1<k<l, ~M(l)\subset M(k)$. Hence, 
		$$|M(l)|\leq \left\{\frac{C}{S^{p-1}}\right\}^{\frac{p^*}{p-1}}\frac{1}{(l-k)^{p^*}}|M(k)|^{\frac{p^*}{p}} .$$
By the Lemma 4.1 of \cite{Stampacchia}, we can guarantee the existence of a $T>0$ independent of $n$ such that $|M(T)|=0$. Therefore, $||u_n||_\infty\leq T.$ \\
		{\it Case 2:} Suppose $v$ is such that
	\begin{eqnarray}
		\int_{\Omega}|\nabla v|^{p-2}\nabla v\cdot\nabla \phi &<& \lambda\int_{\Omega}\frac{\phi}{v^{\gamma}}~\forall\lambda\in W_0^{1,p}(\Omega),\, \phi\geq 0.\label{weak_eqn1}
		\end{eqnarray}
		Let $\lambda>0$. Choose $v=(\frac{1}{\lambda})^{\frac{1}{\gamma+p-1}}w$. We can see that $v$ satisfies 
		$$\int_{\Omega}|\nabla v|^{p-2}\nabla v.\nabla\phi < \int_\Omega\frac{\phi}{v^\gamma},~\forall\phi\in W_0^{1,p}(\Omega), ~ \phi>0.$$
		Therefore, using the result from {\it Case 1}, for $\lambda=1$, we have 
		$||v||_\infty\leq T $, which implies that $\|u_n\|_\infty \leq R {\lambda}^{\frac{1}{\gamma+p-1}}.$ Hence, $(u_n)$ is uniformly bounded in $\Omega$. Finally, on using a result due to Lieberman \cite{Lieberman}, helps us to conclude that $u_n\in C^1(\Omega)$, $\forall n\in\mathbb{N}$. 
	\end{proof}
	\begin{lemma}\label{l3}
	Every bounded nontrivial solution $v$ of the problem $-\Delta_p u=g(u) +\mu_n$ in $\Omega$, is uniformly bounded below in $L^\infty(\Omega)$, i.e. $\|v\|_{\infty}>\delta$, for some $\delta>0$.
	\end{lemma}
	\noindent We instead first prove the following lemma.
	\begin{lemma}\label{ll3}
		Every bounded nontrivial solution $u$ of the problem $-\Delta_p u=g(u)$ in $\Omega$, is uniformly bounded below in $L^\infty(\Omega)$, i.e. $\|u\|_{\infty}>\delta$, for some $\delta>0$.
	\end{lemma}
	\begin{proof}
	
		Let us consider a sequence of nontrivial solutions $(u_m)$ such that $||u_m||_{\infty}\rightarrow 0$ as $m\rightarrow \infty$. Then we can define $w_m(x)=u_m(x)||u_m||_{\infty}^{-1}.$ Clearly, $\|w_m\|_{\infty}=1$. As $u_m$ satisfies $-\Delta_p u=g(u)$, we have
		\begin{align}
		\Delta_p w_m & = \Delta_p (u_m(x)||u_m||_{\infty}^{-1}) \nonumber\\
		& = \nabla(|\nabla (u_m(x)||u_m||_{\infty}^{-1})|^{p-2} \nabla (u_m(x)||u_m||_{\infty}^{-1}))\nonumber\\
		&= \Delta_p u_m||u_m||_{\infty}^{1-p}\nonumber\\
		&= g(u_m)||u_m||_{\infty}^{1-p}\nonumber\\ & \leq C u_m^q||u_m||_{\infty}^{1-p} \nonumber\\
		&\leq C w_m^q||u_m||_{\infty}^{1-p+q} \nonumber\\
		&= f_m.\nonumber
		\end{align}	
		Now for very large $m$, these $f_m$'s are uniformly bounded in $L^{\infty}(\Omega)$. So, $||w_m||_{C^{1,\beta}(\bar\Omega)}\leq M$ for some $\beta\in (0,1)$, by regularity results in \cite{Tolksdorf}, where $M$ is independent of $m$. Hence, by the  Ascoli-Arzela theorem, the sequence $(w_m)$ converges uniformly to $w$ in $C_0^1(\Omega)$. This implies $w=0$. But with the consideration of the Lemma 1.1 of \cite{Azizieh Cle}, we have a unique solution $w$ in $C_0^1(\Omega)$, which contradicts the fact that $||w_m||_{\infty}=1$. Hence, there exists $\delta>0$ such that $\|u\|_{\infty}>\delta$.
	\end{proof}
\noindent{\it{Proof of Lemma \ref{l3}.}} Since $\mu_n\geq 0$, then the solutions of the problem in Lemma \ref{l3} are supersolutions of the problem in Lemma \ref{ll3}. Therefore, if $v$ and $u$ are  solutions of the problem in Lemma \ref{l3} and Lemma \ref{ll3} respectively, then $||v||_\infty\geq ||u||_\infty>\delta>0$, for some $\delta>0$.
	\begin{lemma}\label{l4}
		There exists a $\bar{\lambda}>0$ such that the following problem 
		\begin{eqnarray}\label{ne7}
		-\Delta_p u&=& \cfrac{\lambda}{(u+\frac{1}{n})^\gamma}  +g(u)+\mu_n~\text{in}~\Omega,\nonumber\\
		u&=&0~\text{on}~\partial\Omega,\\
		u&>&0~\text{in}~\Omega\nonumber
		\end{eqnarray}
		does not have any weak solution $u\in W_0^{1,p}(\Omega)$ for $\lambda\geq\bar{\lambda}$.
	\end{lemma}
	\begin{proof}
		Let $\lambda_1$ be the first eigenvalue of the operator $-\Delta_p$ and its corresponding eigenfunction $\phi_1\geq 0$ be such that
		\begin{eqnarray}
		-\Delta_p \phi_1&=& \lambda_1\phi_1^{p-1}~\text{in}~\Omega\nonumber,\\
		\phi_1&=&0~\text{on}~\partial\Omega.\nonumber
		\end{eqnarray}
		Its weak formulation with the test function $\phi=\phi_1$ is given by
		$$\int_{\Omega}|\nabla \phi_1|^p=\lambda_1\int_{\Omega}\phi_1^p.$$
		Let $u_n$ be the weak solution of $\eqref{e7}$, then by the strong maximum principle \cite{Vazquez}, we get $\frac{\phi_1^p}{u_n^{p-1}}\in W_0^{1,p}(\Omega)$. On applying the Picone's Identity (Theorem 2.1 in \cite{Bal}), we have
		\begin{align*}
		&\int_{\Omega} |\nabla\phi_1|^p dx-\int_{\Omega} \nabla(\frac{\phi_1^p}{u_n^{p-1}})|\nabla u_n|^{p-2} \nabla u_n dx  \geq 0\\
		 \Rightarrow &\int_{\Omega} \lambda_1\phi_1^p- \frac{\phi_1^p}{u_n^{p-1}}\cfrac{\lambda}{(u_n+\frac{1}{n})^\gamma}-g(u_n) \frac{\phi_1^p}{u_n^{p-1}}-\mu_n \frac{\phi_1^p}{u_n^{p-1}} dx \geq 0\\
		  \Rightarrow&\int_{\Omega} \left(\lambda_1 u_n^{p-1}- \lambda (u_n+\frac{1}{n})^{-\gamma}-g(u_n)-\mu_n\right)\phi_1^p dx \geq 0.
		\end{align*}
		Consider $\bar{\lambda}$ defined as $\bar{\lambda}=\underset{x\in\Omega}{\max}\cfrac{\lambda_1 u_n^{p-1}- g(u_n)-\mu_n}{(u_n+1)^{-\gamma}}$. Now for every $\epsilon>0$, there exists a $\delta>0$ such that $v^q < \epsilon v^{p-1}, \forall v\in[0,\delta]$. Therefore, $\bar{\lambda}>0$ for some $\epsilon$ and for $\lambda\geq\bar{\lambda}$, we have
		\begin{align}\lambda&\geq \max_{x\in\Omega}\frac{\lambda_1 u_n^{p-1}- g(u)-\mu_n}{(u+1)^{-\gamma}}\nonumber\\
		&\geq \frac{\lambda_1 u_n^{p-1}- g(u)-\mu_n}{(u+\frac{1}{n})^{-\gamma}}\nonumber\\
		&\Rightarrow\left(\lambda_1 u_n^{p-1}-\lambda \left(u+\frac{1}{n}\right)^{-\gamma}-g(u)-\mu_n\right)<0
		\end{align}
		which is a contradiction to our assumption. Hence, for $\lambda\geq\bar{\lambda}$, the problem $\eqref{e7}$ does not possess any solution $u\in W_0^{1,p}(\Omega)$.
	\end{proof}
	\begin{lemma}\label{l5}
		Let $\Omega$ be a strictly convex domain and $u_n$ be a solution of problem $\eqref{e7}$. Then there exists $M>0$, which does not depend on $n$, such that $||u_n||_\infty\leq M$.
	\end{lemma}
	\begin{proof}
	We divide the proof of this lemma into six steps.\\
		\textbf{Step 1 (Uniform H\"{o}pf Lemma)}. Our aim is to show that $\frac{\partial u_n}{\partial n}(x)< c <0$ for any $n\in \mathbb{N}$, where $c$ is some constant which is independent of $n$ but depends on $x$. $\hat{n}$ is the unit outward normal to the boundary $\partial\Omega$ at the point $x$.\\
		Now $\Omega$ satisfies the interior ball condition as it has a $C^2$ boundary, i.e. for some $x_0\in \partial\Omega$, there exists a $B_r(y)\subset\Omega$ such that $\partial B_r(y)\cap \partial\Omega =\{x_0\}$. Let us define $v: B_r(y)\rightarrow \mathbb{R}$ given by $$v(x)=[2^{\frac{N-p}{p-1}}-1]^{-1}r^{\frac{N-p}{p-1}}|x-y|^{\frac{p-N}{p-1}} - [2^{\frac{N-p}{p-1}}-1]^{-1}.$$ 
	We observe that,
		\begin{enumerate}
			\item[(i)] $v(x)= 1$ on $\partial B_{\frac{r}{2}}(y)$ and $v(x)=0$ on $\partial B_r(y)$, and
			\item[(ii)] if $x\in B_r(y)\setminus B_{\frac{r}{2}}(y)$ with $|\nabla v(x)|> c >0$ for some constant $c$ independent of $n$.
		\end{enumerate}
		 Therefore, we have $0< v(x)<1$. Let us define $m=\inf\{u_n(x)|x\in \partial B_{\frac{r}{2}}(y)\}.$ By using the Lemma $\ref{l2}$, we can conclude that $m>0$ and is independent of $n$. on choosing $w=mv$, we see that $w$ satisfies 
		\begin{eqnarray}\label{e8}
		-\Delta_p w&=& 0~\text{in}~ B_r(y)- \overline{B_{\frac{r}{2}}(y)} \nonumber,\\
		w&=&m~\text{if}~ x\in\partial B_{\frac{r}{2}}(y),\nonumber\\
		w&=&0~\text{if} x\in \partial B_r(y)\nonumber.
		\end{eqnarray}
		We have $u_n\geq w$ on the boundary of $B_r(y)- \overline{B_{\frac{r}{2}}(y)}$ and $-\Delta_p w\leq -\Delta_p u_n$ in $\Omega$. Hence, by the weak comparison principle, we have  $u_n\geq w$ in $B_r(y)- \overline{B_{\frac{r}{2}}(y)}$. Since, $u_n(x_0)=w(x_0)=0$, then from the properties of $v$ in (i) and (ii) above, we obtain
		\begin{align}
		\frac{\partial u_n}{\partial \hat{n}}(x_0) & =\lim_{t\rightarrow 0} \frac{u_n(x_0-t\hat{n})}{t} \leq \lim_{t\rightarrow 0}\frac{w(x_0-t\hat{n})}{t}\nonumber\\
		& = \frac{\partial w}{\partial\hat{n}}(x_0)= m \frac{\partial w}{\partial\hat{n}}< -c <0,\nonumber
		~\text{where $c>0$ is independent of $n$.}
		\end{align}
	\textbf{Step 2 (Existence of a neighbourhood of the boundary which does not contain any critical points of $u_n$)}. Let us denote $C(u_n)=\{x\in\Omega:\nabla u_n(x)=0\}$, as the set of critical points of $u_n$. From Step 1, we have $ \frac{\partial u_n}{\partial \eta}<0$ on the boundary. Hence, $dist(\partial\Omega, C(u_n))=b_n>0,~\forall\, n\in \mathbb{N}$ as $\partial\Omega$ and $C(u_n)$ are compact subsets in $\Omega$.\\
	\textbf{Claim:} There exists $\epsilon>0$, independent of $n$, such that $b_n>\epsilon>0$. In other words there exists a neighbourhood $\Omega_\epsilon=\{x\in\Omega: dist(x,\partial\Omega)<\epsilon\}$, such that $C(u_n)\cap\Omega_{\epsilon}=\phi$.\\
	{\it{Proof.}} We prove this by a contrapositive argument. Let there does not exist any such $\epsilon>0$ such that $C(u_n)\cap\Omega_{\epsilon}\neq\phi$. Then there exists $x_n\in C(u_n)$ such that $dist(x_n,\partial\Omega)\rightarrow 0$ as $n\rightarrow\infty$. Therefore, upto a subsequence $x_{n_k}\rightarrow x_0$ and  $x_0\in \partial\Omega$. But from Step 1, we obtain $ \frac{\partial u_n}{\partial \eta}(x_0) <c< 0$. Hence, there exists $l>0$ such that $|\nabla u_n(x)|>\frac{c}{2}$ for $x\in B_l(x_0)\cap\Omega$, where $c$ is independent of $n$. This implies that $B_l(x_0)\cap C(u_n)=\phi$. This is a contradiction, since we can find  $x_{n_0}\in B_l(x_0)\cap\Omega$ such that $\nabla u_{n_0}(x_{n_0})=0$. Hence the claim.
\newpage
\noindent\textbf{Step 3 (Monotonicity of $u_n$)}. Let $e\in \mathbb{S}^{N-1}, \delta\in \mathbb{R}$, then for a fixed $n\in \mathbb{N}$, we define the following
		\begin{enumerate}
		\item[(i)] The hyperplane $\mathbb{L}_{\delta,e}=\{x\in \mathbb{R}^N:x.e=\delta\}$ and $\mathbb{\sigma}_{\delta,e}=\{x\in \mathbb{R}^N:x.e<\delta\}$.
		\item[(ii)] $\hat{x}$ be the reflection of $x$ with respect to the hyperplane $\mathbb{L}_{\delta,e}$ i.e. $\hat{x}=x+2(\delta-x.e)e.$
		\item[(iii)] $a(e)=\underset{x\in \Omega}{\inf} \{x.e\}$ and the reflected cap of $\sigma_{\delta,e}$ with respect to $\mathbb{L}_{\delta,e}$ for any $\delta>a(e)$ denoted as $\hat{\sigma}_{\delta,e}$.
		\item[(iv)] $\hat{\sigma}_{\delta ,e}$ is not internally tangent to $\partial\Omega $ at some point $p\notin\mathbb{L}_{\delta ,e}$.
		\item[(v)] $\hat{n}(x)$ be the unit inward normal to $\partial\Omega$ at $x$, then $\hat{n}(x).e\neq0, \forall x\in \partial\Omega\cap \mathbb{L}_{\delta ,e}$.
		\item[(vi)] $\xi(e)=\{\mu_0>a(e): \forall\delta\in (a(e),\mu_0),  \text{4 and 5 holds}\}.$ and $\bar{\xi}(e)=\sup\{\xi(e)\}.$
		\end{enumerate}
		If $\Omega$ is strictly convex, then the map $e\mapsto \bar{\xi}(e)$ is continuous by Proposition 2 of \cite{Azizieh Lema}. Let us denote $v_n(x)=u_n(\hat{x})$. Considering the strict convexity of $\Omega$ and the property (4), we see that $\hat{\sigma}_{\delta ,e}$ is contained in $\Omega$ for any $\delta\leq \delta_1$ where $\delta_1$ only depends on $\Omega$. Since, $\Delta_p$ is invariant under reflection and both $u_n$ and $v_n$ satisfy equation $\eqref{e7}$ hence both the functions take the same value on the hyperplane $\mathbb{L}_{\delta , e}$. Let us define $\delta_0=\min(\delta_1, \epsilon)$. Also for $x\in \partial\Omega\cap\partial\sigma_{\delta ,e}$, we have $u_n(x)=0$ and $v_n(x)=u_n(\hat{x})>0$ as $\hat{x}\in\Omega$. Therefore,
		\begin{eqnarray}
		-\Delta_p u_n+\cfrac{\lambda}{(u_n+\frac{1}{n})^\gamma}+g(u_n)+\mu_n&=& -\Delta_p v_n+\cfrac{\lambda}{(v_n+\frac{1}{n})^\gamma}+g(v_n)+\mu_n~\text{in}~\sigma_{\delta ,e}\nonumber\\
		u_n&\leq&v_n~\text{on}~\partial\sigma_{\delta ,e}\cap\partial\Omega.\nonumber
		\end{eqnarray}
		Then $u_n\leq v_n$ in $\sigma_{\delta ,e}$ for any $\delta\in(a(e), \delta_0)$, by the comparison principle \cite{Damascelli}. Hence, $u_n$ is nondecreasing for all $x\in\sigma_{\delta_0,e}$ along the $e$-direction.\\
		\noindent\textbf{Step 4} ({\bf Existence of a measurable proper subset of $\Omega$ of nonzero measure on which $u$ is nondecreasing}).  For a fixed $x_0\in\partial\Omega$, let $e=e(x_0)$ be the unit outward normal to $\partial\Omega$ at $x_0$. Then by the results in Step 3, we conclude that $u_n$ is nondecreasing in the direction of $e$ for all $x\in\sigma_{\delta ,e}$ and $a(e)<\delta<\delta_0$. For any $\theta\in\mathbb{S}^{N-1}$ in a small neighbourhood of $e$, the reflection of $\sigma_{\delta ,\theta}$ w.r.t. $\mathbb{L}_{\delta,\theta}$ is a member of $\Omega$, since the domain is strictly convex and hence the sequence $u_n$ will be nondecreasing in the $\theta$ direction. Fix $\delta=\frac{\delta_0}{2}$. Since $\Omega$ is strictly convex, there exists a neighbourhood $\Theta\in\mathbb{S}^{N-1}$ such that $\sigma_{\frac{\delta_0}{2}, e}\subset\sigma_{\delta_0,\theta}$ for all $\theta\in\Theta$. Thus, we can conclude that $u_n$ is nondecreasing in every direction for $\theta\in\Theta$ and for any $x$ with $x\cdot e<\frac{\delta_0}{2}$.\\
		Consider $$\sigma_0=\left\{x\in\Omega : \frac{\delta_0}{8}<x\cdot e<\frac{3\delta_0}{8}\right\}.$$ Obviously, $\sigma_0\subset\sigma_{\frac{\delta_0}{2}, e}$ and $u_n$ is nondecreasing in every direction $\theta\in\Theta$ and $x\in\sigma_0$.	 Choose $\epsilon=\frac{\delta_0}{8}$ and fix a point $x\in\Omega_\epsilon$. Let $x_0$ be the projection of the point $x$ onto $\partial\Omega$. We define $\mathbb{I}_x\subset\sigma_0$ to be the truncated cone having vertex at $x_0-\epsilon e$ and an opening angle $\frac{\theta}{2}$. Then $\mathbb{I}_x$ satisfies the following properties.
		\begin{enumerate}
			\item[(i)] $|\mathbb{I}_x|>k$ for some $k$, where $k$ depends only on $\Omega$ and $\epsilon$,
			\item[(ii)] $u_n(x)\leq u_n(y)$ for all $y\in\mathbb{I}_x$ and $n\in\mathbb{N}$.
		\end{enumerate}
	Then, we have $u_n(x)\leq u_n(x_0-\epsilon e)\leq u_n(y)$, for all $y\in\mathbb{I}_x$.\\
		\noindent\textbf{Step 5 (A boundary `estimate')}. Let us consider the first eigenfunction $\phi_1$ of the $p$-Laplacian eigenvalue problem over $\Omega$. Using the Picone's identity on $\phi_1$, $u_n$ and then applying the strong maximum principle \cite{Vazquez}, we have $\frac{\phi_1^p}{u_n^{p-1}}\in W_0^{1,p}(\Omega)$. Denote $f_n(u_n)=\frac{\lambda}{(u_n+\frac{1}{n})^\gamma}+\mu_n$. Then, we have
		\begin{eqnarray}\label{e9}
		\begin{split}
		\int_\Omega\cfrac{[f_n(u_n)+g(u_n)]\phi_1^p}{u_n^{p-1}}&=\int_\Omega |\nabla u_n|^{p-2}\nabla u_n\cdot\nabla \left(\frac{\phi_1^p}{u_n^{p-1}}\right)\\&\leq\int_\Omega|\nabla \phi_1|^p dx\\
		&\leq C(\Omega).
		\end{split}
		\end{eqnarray}
		Let $\phi_1(z)\geq\xi>0$ for all $z\in\Omega-\Omega_{\frac{\epsilon}{2}}$. Hence, from $\eqref{e9}$, we have
		$$\xi^p \int_{\Omega-\Omega_{\frac{\epsilon}{2}}} \cfrac{[f_n(u_n)+g(u_n)]}{u_n^{p-1}}\leq C(\Omega).$$ This implies
		$$\int_{\mathbb{I}_x}\cfrac{[f_n(u_n)+g(u_n)]}{u_n^{p-1}}\leq\frac{C(\Omega)}{\xi^p}.$$
		Now since, 
		\begin{eqnarray}
		\int_{\mathbb{I}_x}\cfrac{[f_n(u_n)+g(u_n)]}{u_n^{p-1}}\geq\int_{\mathbb{I}_x} g(u_n) u_n^{1-p}(z)dz\geq u_n^{q-p+1}(x)|\mathbb{I}_x|
		\end{eqnarray}
		we have
		$$u_n^{q-p+1}(x)\leq\frac{C_1(\Omega)}{\xi^p},$$ for some constant $C_1>0,$ i.e. $u_n(x)\leq C'$, for all $x\in\Omega_\epsilon$ and for all $n\in\mathbb{N}.$\\
		\noindent\textbf{Step 6 (Blow-up analysis)}. We will show that for every open set, $K\subset\subset\Omega,$ there exists $C_K>0$ such that $\|u_n\|_\infty<C_K,$ for every solution $u_n$ of $\eqref{e7}.$ We will prove it by contrapositive argument. Suppose, there exist a sequence $(u_n)$ of positive solutions of the problem $\eqref{e7}$ and a sequence of points $(Z_n)\subset\Omega$ such that $M_n=u_n(Z_n)=\max\{u_n(x):x\in\bar{K}\}\rightarrow\infty$ as $n\rightarrow\infty.$ Using the boundary estimates one can assume that $Z_n\rightarrow x_0$ as $n\rightarrow\infty$, where $x_0\in \bar{K}$.
		Let $dist(\bar{K},\partial\Omega) =2d$ and $\Omega_d=\{x\in\Omega : dist(x,\Omega)<d\}.$\\
		Let $R_n$ be the sequence of positive real numbers with $R_n^{\frac{p}{q-p+1}}M_n=1.$ Observe that $M_n\rightarrow\infty$ iff $R_n\rightarrow 0$ as $n\rightarrow\infty.$ Define, $w_n:B_{\frac{d}{R_n}(0)\rightarrow\mathbb{R}}$ such that $$w_n(y)=R_n^{\frac{p}{q-p+1}}u_n(Z_n+R_n y).$$ Now $u_n$ has a maximum at $Z_n$, hence we have $\|w_n\|_\infty=w_n(0)=1.$ Since $R_n\rightarrow 0$ there exists $n_0 $ such that $B_R(0)\subset B_{\frac{d}{R_n}}(0)$ for fixed $R>0.$ Again, we have that $w_n$ satisfies the following
		$$\nabla w_n(y)=R_n^{\frac{p}{q-p+1}+1}\nabla u_n(Z_n+R_n y)$$ and 
		$$-\Delta_p w_n(y)=R_n^{\frac{pq}{q-p+1}}[\lambda f_n(u_n(Z_n+R_n y))+R_n^{\frac{-pq}{q-p+1}} w_n^q(Z_n+R_n y) +R_n^{\frac{-pq}{q-p+1}}\mu_n(Z_n+R_n y)].$$ 
		From  Lemma $\eqref{l1}$ and Lemma $\eqref{l3}$, for any $y\in B_R(0)$, we have $Z_n+R_n y\in\bar{\Omega}_d\subset\Omega$ and
		\begin{align}\label{inequa1}
		\begin{split}
		&R_n^{\frac{pq}{q-p+1}}[\lambda f_n(u_n(Z_n+R_n y))+R_n^{\frac{-pq}{q-p+1}} w_n^q(Z_n+R_n y) +R_n^{\frac{-pq}{q-p+1}}\mu_n(Z_n+R_n y)]\leq C(\bar{\Omega}_d),
		\end{split}
		\end{align} for every $n\geq n_0.$ Let us fix a ball $B$ such that $\bar B\subset B_{\frac{d}{R_n}}(0),~\forall~n\geq n_0.$ Then by the interior estimates of Lieberman \cite{Lieberman} and Tolksdorf \cite{Tolksdorf}, we have the existence of a constant $C=C(N,p,B)>0$ and $\beta=\beta(N,p,B)\in(0,1)$ such that
		$$w_n\in C^{1,\beta}(\bar B)~ \text{and}~ ||w_n||_{1,\beta}\leq C.$$
		Using the  Arzela-Ascoli theorem, we guarantee the existence of a function $w\in C^1(\bar B)$ such that there exists a convergent subsequence $w_n\rightarrow w$ in $C^1(\bar B).$ On passing the limit $n\rightarrow\infty,$ we have
		\begin{align*}
		&\int_B|\nabla w|^{p-2}\nabla w\cdot\nabla\phi\geq C\int_B w^q\phi, ~\forall~\phi\in C_c^\infty(B),~ w\in C^1(\bar B),~ w\geq 0 ~\text{on}~\bar B,
		\end{align*} 
		where the constant is obtained from the growth condition over $g$ and the condition in \eqref{inequa1}. Also, we have $||w||_\infty=1.$ Hence, by using the strong maximum principle \cite{Vazquez}, we have $w(x)>0,$ $\forall~x\in B$. Now for a sequence of balls with increasing radius, the Cantor diagonal subsequence converges to $w\in C^1(\mathbb{R}^N)$, on every compact subsets of $\mathbb{R}^N$ and satisfy the following
		\begin{align*}
		&\int_{\mathbb{R}^N}|\nabla w|^{p-2}\nabla w\cdot\nabla\phi\geq C\int_{\mathbb{R}^N} w^q\phi; ~\forall~\phi\in C_c^\infty(\mathbb{R}^N),~w\in C^1(\mathbb{R}^N),~ w>0 ~\text{on}~\mathbb{R}^N.
		\end{align*}
		This contradicts the Theorem $\ref{t1}$.
	\end{proof}
	\begin{lemma}\label{l6}
		For a strictly convex domain $\Omega$, there exists  $\bar\lambda>0$ such that for $0<\lambda<\bar{\lambda}$ and $\gamma>0$ atleast two solutions (say $u_n,v_n$) exist for the problem $\eqref{e7}$ in $W_{loc}^{1,p}(\Omega)$.
	\end{lemma}
		\begin{proof}
			We define $\bar{J_\lambda}:C(\bar{\Omega})\rightarrow C(\bar{\Omega})$ by $$\bar{J_\lambda}(u)=(-\Delta_p)^{-1}\left(\frac{\lambda}{(u+\frac{1}{n})^\gamma}+g(u)+\mu_n\right), ~\lambda\geq0.$$
			Now equation $\eqref{e7}$ can be written as $u=\bar{J_\lambda}(u)$. The map $\bar{J_\lambda}$ is compact since, we know $(-\Delta_p)^{-1}$ is a compact operator on $C(\bar{\Omega})$. So, we assume the map $\bar{J_\lambda}$ is also compact. For $0<\lambda<\bar{\lambda}$, we have $(u_n)$ as solutions to the problem $\eqref{e7}$ and $||u_n||_\infty\leq M$, using Lemma $\ref{l4}$ and Lemma $\ref{l5}$. Let us define, $S_1=\{u\in C(\bar{\Omega}):u\geq0~ \text{in}~ \Omega\}$, $\bar{J_0}:S_1\rightarrow S_1$ by $ \bar{J_0}(u)=(-\Delta_p)^{-1}(g(u)+\mu_n)$ and $G:\bar{B_R}\times [0,\infty)\rightarrow S_1$ such that $G(u,\lambda)=\bar{J_\lambda}$.\\
			$\textit{\bf Claim 1.}$ There exists a supersolution to the problem $\eqref{e7}$.\\
		{\it Proof.} Let us define, $N(r)=\frac{1}{3}\left((\frac{r}{R})^{\gamma+p-1}-Cr^{\gamma+q}\right)$, for $r\in[0,\infty)$ where $R$ is the bound used in Lemma $\ref{l2}$ and $C>0$ is the constant used in the growth condition of $g$ and $\eta=\underset{0\leq r\leq\beta_0}{\max}N(r)$, where $\beta_0=\frac{1}{2} (2q-2p+3)^{\frac{1}{p-q-1}}R^{\frac{\gamma+p-1}{p-q-1}}.$\\
			Observe that, $N(r)>0$ for $r\in(0,\beta_1)$, where $\beta_1\in(0,\min{(\gamma,\beta_0)}).$ Now applying the intermediate value property of continuous functions, we get that there exists a $\beta_2\in(0,\beta_1)$ such that $N(\beta_2)=\lambda_0.$ \\
			Denote $\lambda^*=\left(\frac{\beta_2}{R}\right)^{\gamma+p-1}$. So,
			\begin{align}
			\lambda_0=N(\beta_2) & =  \frac{1}{2}\left(\lambda^*-C\beta_2^{\gamma+q}\right)\nonumber\\
			\lambda^*> \lambda_0+\beta_2^{\gamma+q} & =\lambda_0+C[R(\lambda^*)^{\frac{1}{\gamma+p-1}}]^{\gamma+q}\nonumber.
			\end{align}
			Let $u_{n,\lambda^*}$ satisfy $\eqref{e7}$. Then for $n\geq n_0$, we have 
			\begin{align}
			\lambda^* & >\lambda_0 + C\left(\|u_{n,\lambda^*}\|_\infty\right)^q
			\left(\|u_{n,\lambda^*}\|+\frac{1}{n}\right)^\gamma \nonumber	\\&
			>\lambda + C\left(u_{n,\lambda^*}\right)^q
			\left(u_{n,\lambda^*}+\frac{1}{n}\right)^\gamma, ~\text{for}~ \lambda\leq\lambda_0 \\&
			> \lambda + g(u_{n,\lambda^*}) \left(u_{n,\lambda^*}+\frac{1}{n}\right)^\gamma.
			\end{align}
			Hence, 
			$$-\Delta_p u_{n,\lambda^*}=\frac{\lambda^*}{(u_{n,\lambda^*}+\frac{1}{n})^\gamma}+\mu_n> \frac{\lambda}{(u_{n,\lambda^*}+\frac{1}{n})^\gamma}+\mu_n+g(u_{n,\lambda^*}), ~\text{for}~ \lambda\leq\lambda^* ~\text{and}~ n\geq n_0.$$
			Therefore, $u_{n,\lambda^*}\in C^{1,\alpha}(\bar{\Omega})$ is a positive supersolution for some $\alpha>0$ and $u_{n,\lambda^*}$ is a supersolution of 
			\begin{eqnarray}
			\label{mideq1}
			-\Delta_p u&=&\frac{\lambda}{(u+\frac{1}{n})^\gamma}+g(u)+\mu_n,\nonumber\\
			u&=& 0~\text{on}~\partial\Omega,
			\end{eqnarray}
			with $\|u_{n,\lambda^*}\|_\infty\leq \beta_2$.\\
				$\textit{\bf Claim 2.}$ Problem $\eqref{e7}$ possesses a unique solution.\\
				To prove the Claim 2, we define, $$f_n(x,r)=\frac{\lambda(r+\frac{1}{n})^{-\gamma}+g(r)}{r^{p-1}}, ~\text{for}~ r\in [0,\infty).$$ 
				Now the derivative of $f_n$ w.r.t $r$ is given by
				\begin{align}
					f_n^\prime(x,r) &=\frac{1}{r^p}\left[\frac{\lambda\{(1-p-\gamma)r+\frac{1-p}{n}\}}{(r+\frac{1}{n})^{1+\gamma}}\right]+\frac{rg^\prime(r)-g(r)(p-1)}{r^p}\nonumber\\&
					<\frac{1}{r^p}\left[\frac{\lambda[(1-p-\gamma)r+\frac{1-p}{n}]}{(r+\frac{1}{n})^{1+\gamma}}\right]+ (q-p+1)r^{q-p}.\nonumber
				\end{align}
As the function $r^q(r+\frac{1}{n})^{1+\gamma}$ is convex, so there exists a unique $C_n>0$, which is increasing with respect to $\lambda$ such that 
				$$\lambda\left[(p+\gamma-1)C_n+\frac{p-1}{n}\right]> (q-p+1)C_n^{q}(C_n+\frac{1}{n})^{1+\gamma}.$$ Now for $r\leq C_n$, we have 
				$$(q-p+1)r^{q}(r+\frac{1}{n})^{1+\gamma}\leq\lambda\left[(p+\gamma-1)r+\frac{p-1}{n}\right].$$
				Hence, $f_n^\prime(x,r)<0$. Consider $$F_n(x,r)= \frac{\lambda(r+\frac{1}{n})^{-\gamma}+g(r)+\mu_n}{r^{p-1}}, ~\text{for}~ r\in [0,\infty).$$
				Clearly, $F_n^\prime(x,r)=f_n^\prime(x,r)-\frac{\mu_n(p-1)}{r^p}<0.$ Therefore, $F_n$ is decreasing and using the result of D\'{i}az-Sa\'{a} \cite{Diaz}, we guarantee that the problem $\eqref{e7}$ has unique solution and $||u_n||_\infty\leq C_n$.\\
				Thus, we have $\beta_2\leq\delta_0$. So, $$\frac{q-p+1}{\gamma+p-1}\beta_2^{\gamma+q}<\lambda_0,~\text{for}~ \gamma>1.$$
				Choose $$\lambda_m=\frac{\{(q-p+1)(\beta_2+\epsilon)^{q}-\mu_m(p-1)\} (\beta_2+\epsilon+\frac{1}{m})^{1+\gamma}}{(p+\gamma-1)(\beta_2+\epsilon)+\frac{p-1}{m}}<\lambda_0,$$	
				then for all $n\geq m$, we have $C_n(\lambda_0)\geq C_n(\lambda_n)=\beta_2+\epsilon.$ So, $||u_n||_\infty\leq\beta_2+\epsilon.$\\
			We can see that using Lemma $\ref{l3}$, Lemma $\ref{l4}$ and Lemma $\ref{l5}$, $\bar{J_0}$ and $G$ satisfy all the conditions of Lemma \eqref{figueiredo} taken from \cite{Figueiredo} for some $0<r<\beta_2<R$. Since $\beta_2<\alpha$, $(I-\bar{J_0})(u)$ has no solution on $\partial B_r$. Now considering Lemma \ref{l4} and using Lemma \ref{arcoya lemma} of \cite{Ambrosetti}, we can obtain a continuum $A_n\subset A=\{(\lambda,u)\in[0,\bar{\lambda}]\times C(\bar{\Omega}): u-\bar{J_\lambda}(u)=0\}$ such that
			\begin{equation}\label{eq2}
			A_n\cap (\{0\}\times B_r)\neq \phi,~ A_n\cap (\{0\}\times (B_R-B_r))\neq \phi.
			\end{equation}
			Next, we define $F:[0,\lambda_0]\rightarrow C_0^{1,\alpha}(\bar{\Omega})$ a continuous map such that $F(\lambda)=u_{n,\lambda^*}$. Using Lemma $\ref{l8}$, we conclude that there exists $u_n\in A_n^{\lambda_0}=\{u\in C(\bar{\Omega}):(\lambda_0,u)\in A_n\}$ such that $0<u_n<u_{n,\lambda^*}$. We have $||u_{n,\lambda^*}||_\infty\leq\beta_2$ and hence $||u_n||_\infty\leq ||u_{n,\lambda^*}||_\infty\leq \beta_2.$\\
			We have $A_n\cap(\{0\}\times(B_R-B_r))\neq\phi$ by equation $\eqref{eq2}$. Hence, for $n\geq\max(n_0,m)$, there exists $v_n$ such that $||v_n||_\infty\geq \beta_2+\epsilon.$ For $\lambda=\lambda_0$ we have at least two solutions $u_n$ and $v_n$ to the problem $\eqref{e7}$. As $\lambda_0<\bar{\lambda}$ is arbitrary, it concludes the proof.
			\end{proof}
	\begin{theorem}
	Given $\gamma>0$ there exists $\bar{\lambda}>0$ such that the problem \eqref{e0} admits atleast two solutions $u$, $v$ in $W_{loc}^{1,p}(\Omega)$, provided $\Omega$ is strictly convex with $1<p<N$, $p-1<q<\frac{p(N-1)}{N-p}-1$ and for $0<\lambda<\bar{\lambda}$.
	\end{theorem}
\begin{proof}  From the above Lemma $\ref{l6}$, we can conclude the existence of atleast two solutions $u_n$ and $v_n$ of the problem $\eqref{e7}.$ Also for a suitable choice of $c>0, ~~\underbar u=(c\phi_1+n^{\frac{1+p-\gamma}{p}})^{\frac{p}{\gamma +p-1}}-\frac{1}{n}$ will be a weak subsolution to the problem $\eqref{e2}$ for $\lambda=\lambda_0.$\\
Again, using $\frac{\lambda_0}{(r+\frac{1}{n})^\gamma}\leq \frac{\lambda_0}{(r+\frac{1}{n})^\gamma}+g(r)+\mu_n$ for all $r\geq 0$ we can conclude that each solutions of the problem $\eqref{e7}$ with $\lambda=\lambda_0$ is a weak supersolution of $\eqref{e2}.$ Now by the strong comparison principle \cite{Guedda}, we have 
\begin{align}\label{e10}
\bar u\leq u_{n,\lambda_0}\leq u_n\leq\beta_2,~ \bar u\leq u_{n,\lambda_0}\leq v_n ~\text{and}~ ||v_n||_\infty\geq\beta_2+\epsilon.
\end{align}
Let us take $z_n=u_n$ or $v_n$, then from $\eqref{e10}$ and the Lemma $\ref{l5}$ we have, $$\bar u\leq z_n\leq M,$$ where $M$ is independent of $n.$ By using the strong comparison principle \cite{Guedda} and Lemma $\ref{l2}$, we have
\begin{align}\label{e11}
\forall\,K\subset\subset\Omega,~\exists~C_K ~\text{such that}~ z_n\geq C_K>0 ~\text{in}~ K,~\forall\, n\in\mathbb{N}	.
\end{align}
\textbf{Claim.} $(z_n)$ is bounded in $W_{loc}^{1,p}(\Omega).$
{\it Proof.} Consider $z_n\phi^p$ as a test function in the equation $\eqref{e7}$ for $\phi\in C_0^1(\Omega)$, then we get
	\begin{align*}
	\int_\Omega|\nabla z_n|^p\phi^p=-p\int_\omega\phi^{p-1}z_n|\nabla z_n|^{p-2}\nabla\phi\cdot\nabla z_n+\int_\Omega\frac{\lambda_0 z_n\phi^p}{(z_n+\frac{1}{n})^\gamma}+\int_\Omega z_ng(z_n)\phi^p+\int_\Omega z_n\mu_n
	\end{align*}
	By using the modified Young's inequality we have, $\int_\Omega|\nabla z_n|^p\phi^p\leq C_\phi~\forall~n\in\mathbb{N}$, where $C_\phi$ is a constant depending only on $\phi.$ Hence, $z_n\in W_{loc}^{1,p}(\Omega)$ and there exists $z\in W_{loc}^{1,p}(\Omega)\cap L^\infty(\Omega)$ such that $z_n\rightarrow z$ a.e upto a subsequence and $z_n\rightarrow z$ weakly in $W^{1,p}(K)$ for all $K\subset\subset\Omega.$ From the Theorem 4.4 of \cite{Canino}, $\int_\Omega|\nabla u_n|^{p-2}\nabla u_n\cdot\nabla\phi$ converges to $\int_\Omega|\nabla u|^{p-2}\nabla u\cdot\nabla\phi.$\\ Again, by using dominated convergence theorem, we have $$\lim_{n\rightarrow\infty}\int_\Omega\left(\frac{\lambda_0\phi}{(z_n+\frac{1} {n})^\gamma}+\phi~ g(z_n)\right)dx=\lambda_0\int_\Omega\frac{\phi}{z^\gamma}dx+ \int_\Omega\phi~ g(z)dx$$
	Since, $||u_n||_\infty\leq\beta_2,~ ||v_n||_\infty\geq\beta_2+\epsilon>\beta_2
	$ and $u_n\rightarrow u, v_n\rightarrow v$, we have the existence of two distinct solutions $u$ and $v.$
\end{proof}
\noindent We will now prove the existence result of the problem $\eqref{e0}$.
\section{Existence result}\label{main}
\subsection{The case of $\gamma < 1.$}\label{sub1}
Let us consider the problem in $\eqref{e7}$ for the case of $\gamma<1$.
\begin{lemma}\label{ml1}
	Let $u_n$ be a solution of $\eqref{e7}$ with $\gamma<1$. Then $(u_n)$ is bounded in $W_0^{1,r}(\Omega)$ for every $r<\frac{N(p-1)}{N-1}$.
\end{lemma}
	\begin{proof}
		We will prove the boundedness of ($\nabla u_n$) in the Marcinkiewicz space $\mathcal{M}^\frac{N(p-1)}{N-1}(\Omega)$. For this, let us take $\varphi =T_k(u_n)$ as a test function in the weak formulation $\eqref{weak}$ and we have
		\begin{equation}\label{me1}
		\int_\Omega |\nabla T_k(u_n)|^p   = \int_\Omega \frac{\lambda}{(u_n+ \frac{1}{n})^\gamma}T_k(u_n) +\int_\Omega g(u_n)T_k(u_n) + \int_\Omega T_k(u_n) \mu_n.
		\end{equation}
		Observe \hspace{0.3cm} $$\frac{T_k(u_n)}{(u_n+\frac{1}{n})^{\gamma}}\leq\frac{u_n}{(u_n+\frac{1}{n})^{\gamma}}=\frac{u_n^{\gamma}}{(u_n+\frac{1}{n})^{\gamma}u_n^{\gamma-1}}\leq u_n^{1-\gamma}$$ and
	    $$\int_\Omega T_k(u_n) \mu_n \leq k||\mu_n||_{L^1(\Omega)} \leq Ck. $$
		Therefore, we have, 
		\begin{equation}\label{me2}
		\int_\Omega |\nabla T_k(u_n)|^p  \leq Ck.
		\end{equation}
	Now consider the following set inclusion
		\begin{align*}
		\{|\nabla u_n|\geq  t\} & = \{|\nabla u_n|\geq  t,u_n< k\} \cup \{|\nabla u_n| \geq t,u_n \geq  k\}
		\\& \subset \{|\nabla u_n|\geq  t,u_n <k\} \cup \{u_n \geq k\}\subset \Omega.
		\end{align*}
		With the help of the subadditivity property of Lesbegue measure $m$ we have,
		\begin{equation}\label{me3}
		m(	\{|\nabla u_n|\geq t\}) \leq m(\{|\nabla u_n|\geq t,u_n< k\}) + m(\{u_n \geq  k\}).
		\end{equation}
		By the Sobolev inequality, we have
		\begin{equation}\label{sobo}
		\frac{1}{\lambda_1}\left(\int_\Omega |T_k(u_n)|^{p^*}\right)^{\frac{p}{p^*}}\leq \int_{\Omega}|\nabla T_k(u_n)|^p \leq Ck
		\end{equation}
		  where $\lambda_1$ is the first eigenvalue of the $p$-Laplacian operator. Now, on restricting the left hand side of the integral \eqref{sobo} on $I=\left\lbrace x\in\Omega:u_n\geq k \right\rbrace$, such that $T_k(u_n)=k$, we  obtain 
		  \begin{align*}
		 & k^pm(\{u_n\geq k\})^{\frac{p}{p^*}}\leq Ck\\\Rightarrow~&m(\{u_n\geq k\})\leq \frac{C}{k^\frac{N(p-1)}{N-p}},~ \forall\,k\geq1.
		  \end{align*}
		Hence, $(u_n)$ is bounded in $\mathcal{M}^{\frac{N(p-1)}{N-p}}(\Omega)$.\\
		Similarly on restricting $\eqref{sobo}$ on $I^{'}=\{|\nabla u_n|\geq  t,u_n< k\}$, we have
		\begin{center}
			$m(\{|\nabla u_n|\geq  t,u_n< k\})\leq \frac{1}{t^p}\int_\Omega |\nabla T_k(u_n)|^p\leq \frac{Ck}{t^p}, \forall k>1. $
		\end{center}
		\noindent Now $\eqref{me3}$ becomes
		$$	m(	\{|\nabla u_n|\geq t\}) \leq m(\{|\nabla u_n|\geq t,u_n< k\}) + m(\{u_n \geq k\})\leq \frac{Ck}{t^p} + \frac{C}{k^\frac{N(p-1)}{N-p}}, \forall k>1.$$
		Let us choose, $k=t^{\frac{N-p}{N-1}}$ and hence we get $$ m(\{|\nabla u_n|\geq t\})\leq \frac{C}{t^\frac{N(p-1)}{N-1}}, \hspace{0.2cm} \forall\,t\geq 1.$$
		We have proved that $(\nabla u_n)$ is bounded in $\mathcal{M}^{\frac{N(p-1)}{N-1}}(\Omega)$. This implies by property $\eqref{marcin}$ that $(u_n)$ is bounded in $W_0^{1,r}(\Omega)$, for every $r<\frac{N(p-1)}{N-1}$.$\vspace{0.1cm}$
		\end{proof}
\begin{theorem}\label{mt1}
	Let $\gamma < 1$. Then there exists a weak solution $u$ of $\eqref{e0}$ in $W_0^{1,r}(\Omega)$ for every $r<\frac{N(p-1)}{N-1}$.
\end{theorem}
	\begin{proof}	Lemma $\ref{ml1}$, implies that there exists $u$ such that a subsequence of $u_n$ converges weakly to $u$ in $W_0^{1,r}(\Omega)$, for every $r<\frac{N(p-1)}{N-1}$. This implies that for $\varphi$ in $C_c^1(\Omega)$
		$$\lim_{n\rightarrow +\infty} \int_{\Omega} \nabla u_n . \nabla\varphi = \int_{\Omega}\nabla u .\nabla\varphi.$$
		Also due to the compact embeddings we can assume that $u_n$ converges to $u$ both strongly in $L^1(\Omega)$ and a.e. in $\Omega$. Thus, taking $\varphi$ in $C_c^1(\Omega)$, we get,
		\begin{align*}
		0 & \leq \big|\frac{\lambda}{(u_n+\frac{1}{n})^\gamma}\varphi\big|
		\\& \leq C\lambda ||\varphi||_{L^\infty(\Omega)}
		\end{align*}
	    This is sufficient to apply the dominated convergence theorem to obtain
		$$\lim_{n\rightarrow +\infty} \int_{\Omega}  \frac{\lambda}{(u_n+\frac{1}{n})^\gamma}\varphi = \int_{\Omega}\frac{\lambda}{u^\gamma}\varphi.$$
		 Further, since $(u_n)$ is bounded in $W_0^{1,r}(\Omega)$, we have by the compact embedding that $u_n\rightarrow u$ in  $L^r(\Omega)$. By the same standard argument, there exists a subsequence that converge to $u$ uniformly except on a set of arbitrarily small Lebesgue measure. Since, by the hypothesis $g$ is continuous, the limit $n\rightarrow \infty$ can be passed on. On applying a similar argument as in step 4 of the Theorem 3.2 in \cite{Oliva}, we have a.e. convergence of the $\nabla u_n$ towards $\nabla u$ that follows in a standard way by proving that $\nabla T_k (u_n)$ goes to $\nabla T_k (u)$, in $L_{loc}^{r}(\Omega)$ for $r<p$, for every $k>0$. Finally, we can pass the limit $n\rightarrow\infty$ in the last term of $\eqref{weak}$ involving $\mu_n$.
		This concludes the proof of the result as it is easy to pass to the limit in $\eqref{weak}$. Therefore, we obtain a weak solution of $\eqref{e0}$ in $W_0^{1,r}(\Omega)$ for every $r<\frac{N(p-1)}{N-1}$.
	\end{proof}
\subsection{The case of $\gamma\geq1.$}\label{sub2}
Due to the strong singularity we can hold some local estimates on $u_n$ in the Sobolev space. We shall give global estimates on $T_k^{\frac{\gamma+p-1}{2}}(u_n)$ in $W_0^{1,2}(\Omega)$ with the aim of giving sense, at least in a weak sense, to the boundary values of $u$.
\begin{lemma}\label{ml2}
	Let $u_n$ be a solution of $\eqref{e7}$ with $\gamma\geq1$. Then $T_k^{\frac{\gamma+p-1}{p}}(u_n)$ is bounded in $W_0^{1,p}(\Omega)$ for every fixed $k>0$.
\end{lemma}
	\begin{proof}
		Consider $\varphi=T_k^\gamma(u_n)$ as a test function in $\eqref{weak}$. We have
		\begin{align}\label{me4}
		\gamma\int_\Omega |\nabla u_n|^{p-2}\nabla u_n&\cdot\nabla T_k(u_n)T_k^{\gamma-1}(u_n)\nonumber\\ &=\int_{\Omega} \frac{\lambda}{(u_n+\frac{1}{n})^\gamma}T_k^\gamma(u_n) +\int_{\Omega} g(u_n)T_k^\gamma(u_n) + \int_{\Omega}T_k^\gamma(u_n)\mu_n.
		\end{align}
		We can estimate the term on the left hand side of $\eqref{me4}$ as,
		\begin{equation}\label{me5}
			\gamma\int_\Omega |\nabla u_n|^{p-2}\nabla u_n\cdot\nabla T_k(u_n)T_k^{\gamma-1}(u_n)=\gamma\int_\Omega |\nabla T_k^{\frac{\gamma+p-1}{p}}(u_n)|^p.
		\end{equation}
		As $\frac{T_k^\gamma(u_n)}{(u_n+\frac{1}{n})^\gamma}\leq \frac{u_n^\gamma}{(u_n+\frac{1}{n})^\gamma}\leq 1$, the term on the right hand side of $\eqref{me4}$ can be estimated as,
		\begin{align}\label{me6}
		\int_{\Omega} \frac{\lambda}{(u_n+\frac{1}{n})^\gamma}T_k^\gamma(u_n) +\int_{\Omega} g(u_n)T_k^\gamma(u_n) &+ \int_{\Omega}T_k^\gamma(u_n)\mu_n\nonumber\\ 
		& \leq  C\lambda k^\gamma +C\int_{\Omega} u_n^qT_k^\gamma (u_n)+ k^\gamma \int_{\Omega} \mu_n\nonumber 
		\\&  \leq  C\lambda k^\gamma +CM k^\gamma+ k^\gamma \int_\Omega \mu_n\nonumber
		\\& \leq C(k,\gamma)k^\gamma. 
		\end{align}
		On combining the previous inequalities $\eqref{me5}$ and $\eqref{me6}$ we get
		\begin{equation}\label{me7}
		\int_{\Omega} |\nabla T_k^{\frac{\gamma+p-1}{p}}(u_n)|^p \leq Ck^\gamma
		\end{equation}
		then,  $\left(T_k^{\frac{\gamma+p-1}{p}}(u_n)\right)$ is bounded in $W_0^{1,p}(\Omega)$ for every fixed $k>0$.
	\end{proof} 
\noindent Now, so as to pass to the limit $n\rightarrow\infty$ in the weak formulation  $\eqref{weak}$, we require to prove some local estimates on $u_n$. We first prove the following.
\begin{lemma}\label{ml3}
	Let $u_n$ be a solution of $\eqref{e7}$ with $\gamma\geq1$. Then ($u_n$) is bounded in $W_{loc}^{1,r}(\Omega)$ for every $r<\frac{N(p-1)}{N-1}$.
\end{lemma}
	\begin{proof}
		We prove the theorem in two steps.\\
		$\boldmath{\text{\bf Step 1.}}$ We claim that $\left(G_1(u_n)\right)$ is bounded in  $W_0^{1,r}(\Omega)$ for every $r<\frac{N(p-1)}{N-1}$.\\
		We can see that $G_1(u_n)=0$ when $0\leq u_n\leq 1$, $G_1(u_n)=u_n-1$, otherwise i.e when $u_n>1$. So $\nabla G_1(u_n)=\nabla u_n$ for $u_n>1$.\\
		Now, we need to show that $\left(\nabla G_1(u_n)\right)$ is bounded in $\mathcal{M}^{\frac{N(p-1)}{N-1}}(\Omega)$, where  $\mathcal{M}^{\frac{N(p-1)}{N-1}}(\Omega)$ is the Marcinkiewicz space. Then we have
		\begin{align*}
		\{|\nabla u_n|> t, u_n>1\} & = \{|\nabla u_n|> t,1<u_n\leq k+1\} \cup \{|\nabla u_n| > t,u_n > k+1\}
		\\& \subset \{|\nabla u_n|> t,1<u_n\leq k+1\} \cup \{u_n > k+1\}\subset \Omega.
		\end{align*}
		Hence, 
		\begin{equation}\label{me8}
		m(	\{|\nabla u_n|> t,u_n>1\}) \leq m(\{|\nabla u_n|> t,1<u_n\leq k+1\}) + m(\{u_n > k+1\}).
		\end{equation}
		In order to estimate $\eqref{me8}$ we take $\varphi=T_k(G_1(u_n))$, for $k>1$, as a test function in $\eqref{e7}$.\\
		We observe that $\nabla T_k(G_1(u_n))= \nabla u_n$ only when $1<u_n \leq k+1$, otherwise is zero, and  $T_k(G_1(u_n))=0$ on $\{ u_n\leq 1\} $, we have
		\begin{align*}
		\int_\Omega |\nabla T_k(G_1(u_n))|^p & = \int_\Omega \frac{\lambda}{(u+\frac{1}{n})^\gamma}T_k(G_1(u_n)) + \int_{\Omega} g(u_n)T_k(G_1(u_n))+ \int_{\Omega}T_k(G_1(u_n))\mu_n \\& \leq C\lambda k + Ck\int_\Omega u_n^q +k\int_\Omega\mu_n
		\\&  \leq Ck
		\end{align*} 
		and by restricting the above integral on $I_1={\left\lbrace 1<u_n\leq k+1 \right\rbrace}$, we get
		\begin{align}
		\int_{\left\lbrace 1<u_n\leq k+1 \right\rbrace} |\nabla T_k(G_1(u_n))|^p \nonumber& = \int_{\left\lbrace 1<u_n\leq k+1 \right\rbrace} |\nabla u_n|^p\nonumber  \\& \geq \int_{\left\lbrace |\nabla u_n|>t, 1<u_n\leq k+1 \right\rbrace} |\nabla u_n|^p\nonumber \\&\geq t^p m(\{|\nabla u_n|> t,1<u_n\leq k+1\})\nonumber
		\end{align}
		so that, $$m(\{|\nabla u_n|> t,1<u_n\leq k+1\})\leq \frac{Ck}{t^p} \hspace{0.4cm}\forall k \geq 1.$$
		According to $\eqref{me7}$ in the proof of Lemma $\ref{l2}$, one can see that 
		$$\int_{\Omega} |\nabla T_k^{\frac{\gamma+p-1}{p}}(u_n)|^p \leq Ck^\gamma \hspace{0.2cm} \text{for any}\hspace{0.2cm} k>1.$$
		Therefore, from the Sobolev inequality  $$\frac{1}{\lambda_1}\Bigg(\int_\Omega |T_k^{\frac{\gamma+p-1}{p}}(u_n)|^{p^*}\Bigg)^{\frac{p}{p^*}}\leq \int_{\Omega}|\nabla T_k^{\frac{\gamma+p-1}{p}}(u_n)|^p \leq Ck^{\gamma},$$ where, $\lambda_1$ is the first eigenvalue of the $p$-Laplacian operator.
		Now, if we restrict the integral on the left hand side on $I_2=\left\lbrace u_n> k+1 \right\rbrace_{x\in\Omega} $, on which $T_k(u_n)=k$, we then obtain $$k^{\gamma+p-1}m(\{u_n>k+1\})^{\frac{p}{p^*}}\leq Ck^{\gamma},$$
		so that $$m(\{u_n>k+1\})\leq \frac{C}{k^\frac{N(p-1)}{N-p}},~\forall\, k\geq1.$$
		So, $(u_n)$ is bounded in $\mathcal{M}^{\frac{N(p-1)}{N-p}}(\Omega)$, i.e. $(G_1(u_n))$ is also bounded in $\mathcal{M}^{\frac{N(p-1)}{N-p}}(\Omega)$.\\
		Now $\eqref{me8}$ becomes
		\begin{eqnarray}	m(	\{|\nabla u_n|> t,u_n>1\})& \leq & m(\{|\nabla u_n|> t,1<u_n\leq k+1\}) + m(\{u_n > k+1\})\nonumber\\&\leq &\frac{Ck}{t^p} + \frac{C}{k^\frac{N(p-1)}{N-p}}, \forall k>1.\nonumber
		\end{eqnarray}
		We then choose, $k=t^{\frac{N-p}{N-1}}$ and we get $$ m(\{|\nabla u_n|> t,u_n>1\})\leq \frac{C}{t^\frac{N(p-1)}{N-1}} \hspace{0.2cm} \forall t\geq 1.$$
		We just proved that $(\nabla u_n)=(\nabla G_1(u_n))$ is bounded in $\mathcal{M}^{\frac{N(p-1)}{N-1}}(\Omega)$. This implies by property $\eqref{marcin}$ that $(G_1(u_n))$ is bounded in $W_0^{1,r}$ for every $r<\frac{N(p-1)}{N-1}$.$\vspace{0.1cm}$\\			 
		$\boldmath{\text{\bf Step 2.}}$ We claim that $T_1(u_n)$ is bounded in $W_{loc}^{1,r}(\Omega)$.\\
		We have to examine the behavior of $u_n$ for small values of $u_n$ for each $n$. We want to show that for every $K\subset\subset \Omega$, 
		\begin{equation}\label{me9}
		\int_K |\nabla T_1(u_n)|^p \leq C.
		\end{equation}
		We have already proved that $u_n\geq C_K>0$ on $K$ in Lemma $\ref{l2}$. We will use $\varphi=T_1^\gamma(u_n)$ as a test function in $\eqref{weak}$ to get
		\begin{align}\label{me10}
		\gamma\int_\Omega |\nabla u_n|^{p-2}\nabla u_n&\cdot\nabla T_k(u_n)T_k^{\gamma-1}(u_n)\nonumber\\ &=\int_{\Omega} \frac{\lambda}{(u_n+\frac{1}{n})^\gamma}T_k^\gamma(u_n) +\int_{\Omega} g(u_n)T_k^\gamma(u_n) + \int_{\Omega}T_k^\gamma(u_n)\mu_n \nonumber\\ &\leq C.
		\end{align}
	Now observe that 
		\begin{align}\label{me11}
			\gamma\int_\Omega |\nabla u_n|^{p-2}\nabla u_n\cdot\nabla T_1(u_n) T_1^{\gamma-1}(u_n) &\geq\int_K |\nabla T_1(u_n)|^p T_1^{\gamma-1}(u_n)\nonumber\\ &\geq  C_K^{\gamma-1}\int_{K}|\nabla T_1(u_n)|^p.
			\end{align}
		On combining $\eqref{me10}$ and $\eqref{me11}$ we get $\eqref{me9}$.
		We completed the proof as $u_n=T_1(u_n)+G_1(u_n)$. Hence, ($u_n$) is bounded in $W_{loc}^{1,r}(\Omega)$ for every $r<\frac{N(p-1)}{N-1}$.  
	\end{proof}
\noindent Now, we can finally state and prove the existence result for $\gamma\geq 1$.
\begin{theorem}
	Let $\gamma\geq1$. Then there exists a weak solution $u$ of $\eqref{e0}$ in $W_{loc}^{1,r}(\Omega)$ for every $r<\frac{N(p-1)}{N-1}$.
\end{theorem}
	\begin{proof}
		The proof of this theorem is a straightforward application of the Theorem $\ref{mt1}$ and using the results in Lemma $\ref{ml2}$ and Lemma $\ref{ml3}$.
	\end{proof}
	\begin{center}
		${\textbf{Some Important results}}$
	\end{center}
	Define, $X=\{u\in C_0^{1,\alpha}(\bar{\Omega}):u(x)\geq0~\text{ in}~ \bar{\Omega}\}$ and let $\xi$ is a unit outward normal at $\partial\Omega$, then define $X_0=\{u\in C_0^{1,\alpha}(\bar{\Omega}):u(x)>0 ~\text{and}~ \frac{\partial u}{\partial\xi}(x)<0, ~\forall x\in \partial\Omega\}$. Clearly $X_0$ is the interior of $X$.
	\begin{lemma}\label{l7}
		If $u_1,u\in C_0^{1,\alpha}(\bar{\Omega})$ with $u_1\neq u$ and 
		$$-\Delta_p u_1> \frac{\lambda}{(u_1+\frac{1}{n})^\gamma}  +g(u_1)+\mu_n,$$ $$-\Delta_p u= \cfrac{\lambda}{(u+\frac{1}{n})^\gamma}  +g(u)+\mu_n,$$
		then $(u_1-u)\notin \partial X$.
		\begin{proof}
			We prove this Lemma by contradiction. Suppose $(u_1-u)\in \partial X$. Then $u_1(x)\geq u(x)$. By Strong maximum principle \cite{Guedda}, we can obtain $(u_1-u)\in X_0$. But $X_0\cap\partial X=\phi$, for which we get a contradiction. Therefore, $u_1-u$ does not belong to $\partial X.$
		\end{proof}
	\end{lemma}
	\begin{lemma}\label{l8}
		Assume $I$ is an interval in $\mathbb{R}$ and $A=I\times C_0^{1,\alpha}(\bar{\Omega})$ is a connected set of solutions of $\eqref{e7}$. Define $F:I\rightarrow C_0^{1,\alpha}(\bar{\Omega})$ is continuous such that $F(\lambda)$ is a supersolution to the problem $\eqref{e7}$.\\
		If $u_1\leq F(\lambda_1)$ in $\Omega$, $u_1\neq F(\lambda_1)$ for some $(\lambda_1,u_1)\in A$, then $u< F(\lambda)$ in $\Omega$, $\forall(\lambda,u)\in A$.	
		\begin{proof}
			Let $Z:A\rightarrow C_0^{1,\alpha}(\bar{\Omega})$ is a continuous map such that $Z(\lambda,u)=F(\lambda)-u$. $A$ is connected, so by continuity $Z(A)$ is connected in $C_0^{1,\alpha}(\bar{\Omega})$.\\
			Using Lemma $\ref{l7}$, $F(\lambda_1)-u_1=Z(\lambda_1,u_1)\notin\partial X$. Hence, $Z(\lambda_1,u_1)\in X_0$. So, $Z(A)\subset X_0,$ as $Z(A)$ is connected.\\
			Therefore, $F(\lambda)-u>0$, which implies $F(\lambda)>u$, $\forall(\lambda,u)\in A$. Hence, we get our required result.
		\end{proof}	
			\end{lemma}
		\begin{lemma}\label{arcoya lemma}[Ambrosetti-Arcoya \cite{Ambrosetti}].
		Given $X$ be a real Banach space with $U\subset X$ be open, bounded set. Let $a,b\in \mathbb{R}$ such that the equation $u-T(\lambda,u)=0$ has no solution on $\partial U$ for all $\lambda\in [a,b]$ and that $u-T(\lambda,u)=0$ has no solution in $\overline{U}$ for $\lambda=b$. Also let $U_1\subset U$ be open such that $u-T(\lambda,u)=0$ has no solution in $\partial U_1$ for $\lambda=a$ and $\deg(I-K_a, U_1, 0)\neq0
		.$\\
		Then there exists a continuum $C$ in $\sum=\left\{(\lambda, u)\in[a, b]\times X : u-T(\lambda,u)=0 \right\}$ such that $$C\cap(\{a\}\times U_1)\neq\emptyset~\text{and}~C\cap(\{a\}\times (U-U_1))\neq\emptyset.$$
		\end{lemma}
	\begin{theorem}\label{t1}[Mitidieri-Pohozaev \cite{Mitidieri}].
		If $p-1<q<\frac{N(p-1)}{N-p}, p<N$ and $C>0$, then the problem $$\int_{\mathbb{R}^n} |\nabla u|^{p-2}\nabla u.\nabla\phi\geq C \int_{\mathbb{R}^n} u^q\phi;~\phi\in C_c^\infty (\mathbb{R}^n)$$ does not have any positive solution in $C^1(\mathbb{R}^n)$.
	\end{theorem}
\begin{theorem}\label{figueiredo}[De Figueiredo et al. \cite{Figueiredo}].
Let $C$ be a cone in a Banach space X and $\phi:C \rightarrow C$ be a compact map such that $\phi(0)=0$. Assume that there exists $0<r<R$ such that 
\begin{enumerate}
	\item $x\neq t\phi(x)$ for $0\leq t \leq 1$ and $\|x\|=r$
	\item a compact homotopy $F:\bar{B_R}\times[0,\infty)\rightarrow C$ such that $F(x,0)=\phi(x)$ for $\|x\|=R$, $F(x,t)\neq x$ for $\|x\|=R$ and $0\leq t<\infty$ and $F(x,t)=x$ has no solution for $x\in\bar{B_R}$ for $t\geq t_0$.
\end{enumerate} 
Then if, $U=\{x\in C:r<\|x\|<R\}$ and $B_{\rho}=\{x\in C:\|x\|<\rho\}$ we have $deg(I-\phi,B_R,0)=0$, $deg(I-\phi,B_r,0)=1$ and $deg(I-\phi,U,0)=-1$.x

Juha
\end{theorem}
\section*{Acknowledgement}
Two of the authors, S. Ghosh and A. Panda, thanks for the financial assistantship received to carry out this research work from the Council of Scientific and Industrial Research (C.S.I.R.), Govt. of India and Ministry of Human Resource Development(M.H.R.D.), Govt. of India, respectively. This is also to declare that there are no financial conflict of interest whatsoever.
Finally the authors thank the anonymous referee(s) and the editor(s) for the constructive comments and suggestions.
	
\end{document}